\documentclass{amsart}
\usepackage{bbm}
\usepackage{mathrsfs}

\usepackage{amsfonts}
\usepackage{amsmath}
\usepackage{bigdelim}
\usepackage{multirow}
\usepackage{amssymb}
\usepackage{indentfirst}
\usepackage[dvips]{graphicx}

\usepackage{fancyhdr}
\usepackage{amscd,amssymb,amsmath,graphicx,verbatim}
\usepackage[dvips]{hyperref}
\usepackage[TS1,OT1,T1]{fontenc}

\newtheorem{theorem}{Theorem}[section]
\newtheorem{lemma}[theorem]{Lemma}
\newtheorem{corollary}[theorem]{Corollary}
\newtheorem{proposition}[theorem]{Proposition}

\theoremstyle{definition}
\newtheorem{definition}[theorem]{Definition}
\newtheorem{example}[theorem]{Example}

\theoremstyle{remark}
\newtheorem{remark}[theorem]{Remark}
\numberwithin{equation}{section}

\begin{document}
\title{Approximation of chaotic operators}
\author{Geng Tian}
\address{Geng Tian, Department of Mathematics , Jilin university, 130012, Changchun, P.R.China} \email{tiangeng07@mails.jlu.edu.cn}

\author{Luoyi Shi}
\address{Luoyi Shi, Department of Mathematics , Jilin university, 130012, Changchun, P.R.China} \email{shiluoyi811224@sina.com}

\author{Sen Zhu}
\address{Sen Zhu, Department of Mathematics , Jilin university, 130012, Changchun, P.R.China} \email{zhus@email.jlu.edu.cn}

\author{Bingzhe Hou}
\address{Bingzhe Hou, Department of Mathematics , Jilin university, 130012, Changchun, P.R.China} \email{houbz@jlu.edu.cn}


\date{May. 28, 2009}
\thanks{}
\begin{abstract}
\baselineskip 4mm As well-known, the concept "hypercyclic" in
operator theory is the same as the concept "transitive" in dynamical
system. Now the class of hypercyclic operators is well studied.
Following  the idea of research in hypercyclic operators, we
consider  classes of operators with some kinds of chaotic properties
in this article.

First of all, the closures of the sets of all Li-Yorke chaotic
operators or distributionally chaotic operators are discussed. We
give a spectral description of them and prove that the two closures
coincide with each other. Moreover, both the set of all Li-Yorke
chaotic operators and the set of all distributionally chaotic
operators have nonempty interiors which coincide with each other as
well. The article also includes the containing relation between the
closure of the set of all hypercyclic operators and the closure of
the set of all distributionally chaotic operators. Finally, we get
connectedness of the sets considered above.
\end{abstract}

\maketitle
\section{Introduction}
\baselineskip 5mm We are interested in the dynamical systems induced
by continuous linear operators on Banach spaces. From Rolewicz's
article \cite{Rolewicz}, hypercyclicity is widely studied. In fact,
it coincides with a dynamical property "transitivity". Now there has
been got so many improvements at this aspect, Grosse-Erdmann's and
Shapiro's articles \cite{Grosse,Shapiro} are good surveys.

In his celebrated work \cite{Her,Her2,Her1}, D. A. Herrero studied
the chaotic properties (hypercyclic and Devaney's chaotic) of linear
operators. It is important since it shows that we can study the
chaotic properties of operators in a really operator theory way. As
well-known, it is hard to check whether a topological system be
chaotic or not for a general object. But following Herrero's idea,
we can use the technique of approximation to study the properties of
chaotic operators on Hilbert space under compact or small
perturbation. An interesting result, obtained by D. A. Herrero and
Z. Y. Wang \cite{Her} or K. Chan and J. Shapiro \cite{Chan}, shows
that the identity operator $I$ can be perturbed by a small compact
operator to  be hypercyclic. This stronger result implies that a
small perturbation of a simple operator can be an operator with
complex dynamic properties.

These papers suggest us to consider the following question:

{\bf Question :} Which kinds of operators can be approximated by
chaotic operators?

From the point of approximation, we should consider closure of the
set of all operators satisfying some chaotic property. In this
paper, Li-Yorke chaotic operators and distributionally chaotic
operators will be studied by classical approximation tools developed
in \cite{HER}.

In order to explain the main results, we must introduce some
definitions and properties of chaos and Hilbert space operators.

In 1975, Li and Yorke \cite{L-Y} observed complicated dynamical
behavior for the class of interval maps with period 3. This
phenomena is currently known under the name of Li-Yorke chaos.
Recall that a discrete dynamical system is simply a continuous
mapping $f: X\rightarrow X$ where $X$ is a complete separable metric
space. For $x\in X$, the orbit of $x$ under $f$ is
$Orb(f,x)=\{x,f(x),f^{2}(x),\ldots\}$ where $f^{n}= f\circ f\circ
\cdots \circ f $ is the $n^{th}$ iterate of $f$ obtained by
composing $f$ with $n$ times.

\begin{definition}
$\{x,y\}\subset X$ is said to be a Li-Yorke chaotic pair, if
\vspace*{2mm}
\begin{align*}
\limsup\limits_{n\rightarrow\infty}d(f^{n}(x),f^{n}(y))>0, \quad
\liminf\limits_{n\rightarrow\infty}d(f^{n}(x),f^{n}(y))=0.
\end{align*}
Furthermore, $f$ is called Li-Yorke chaotic, if there exists an
uncountable subset $\Gamma\subseteq X$ such that each pair of two
distinct points in $\Gamma$ is a Li-Yorke chaotic pair.
\end{definition}

In 1994, Schweizer and Sm\'{\i}tal \cite{S-S} gave the definition of
distributional chaos (where it was called strong chaos), which
requires more complicated statistical dependence between orbits than
the existence of points which are proximal but not asymptotic.

For any pair $\{x,y\}\subset X$ and any $n\in \mathbb{N}$, define
distributional function $F^{n}_{xy}:\mathbb{R}\rightarrow [0,1]$:
\begin{equation*}
F^{n}_{xy}(\tau)=\frac{1}{n}\#\{0\leq i\leq n-1:
d(f^{i}(x),f^{i}(y))<\tau\},
\end{equation*}
where $\#\{A\}$ is the cardinality of the set $A$. Furthermore,
define
\begin{align*}
F_{xy}(\tau)=\liminf\limits_{n\rightarrow\infty}F^{n}_{xy}(\tau), \\
F_{xy}^{*}(\tau)=\limsup\limits_{n\rightarrow\infty}F^{n}_{xy}(\tau)
\end{align*}
Both $F_{xy}$ and $F_{xy}^{*}$ are nondecreasing functions and may
be viewed as cumulative probability distributional functions
satisfying $F_{xy}(\tau)=F_{xy}^{*}(\tau)=0$ for $\tau<0$.

\begin{definition}
$\{x,y\}\subset X$ is said to be a distributionally chaotic pair, if
\vspace*{2mm}
\begin{align*}
F_{xy}^{*}(\tau)\equiv 1, \ \ \forall \ \ \tau>0 \quad \text{and}
\quad F_{xy}(\epsilon)=0, \ \ \exists \ \ \epsilon>0.
\end{align*}
Furthermore, $f$ is called distributionally chaotic, if there exists
an uncountable subset $\Lambda\subseteq X$ such that each pair of
two distinct points in $\Lambda$ is a distributionally chaotic pair.
Moreover, $\Lambda$ is called a distributionally
$\epsilon$-scrambled set.
\end{definition}

From the definitions, we know distributional chaos implies Li-Yorke
chaos. But the converse implication is not true in general. In
practice, even in the simple case of Li-Yorke chaos, it might be
quite difficult to prove chaotic behavior from the very definition.
Such attempts have been made in the context of linear operators (see
\cite{Duan, Fu}). Further results of \cite{Duan} were extended in
\cite{Opr} to distributional chaos for the annihilation operator of
a quantum harmonic oscillator. Additionally, Distributional chaos
for shift operators were discussed by F. Mart\'{\i}nez-Gim\'{e}nez,
et. al. in \cite{Gim}. More about Li-Yorke chaos and distributional
chaos, one can see \cite{Smital2, Liao1, Liao2, Smital1, Wang}. In a
recent article \cite{Hbz}, B. Hou et. al. introduced a new dynamical
property for linear operators called norm-unimodality which implies
distributional chaos, and obtained a sufficient condition for
Cowen-Douglas operator being distributional chaotic and Devaney's
chaotic. We introduce the definition of norm-unimodality here.
\begin{definition}
Let $X$ be a Banach space and let $T\in B(X)$. $T$ is called
norm-unimodal, if we have a constant $\gamma
>1$ such that for any $m\in\mathbb{N}$, there exists $x_m\in
X$ satisfying
$$
\lim\limits_{k\rightarrow\infty}\|T^kx_m\|=0, \ \ \text{and} \ \ \|
T^ix_m \|\geq \gamma^i\|x_m\|, \ \ i=1,2,\ldots,m.
$$
Furthermore, such $\gamma$ is said to be a norm-unimodal constant
for the norm-unimodal operator $T$.
\end{definition}
Next, we introduce the notations and properties of Hilbert space
operators. Let $H$ be complex separable Hilbert space and denote by
$B(H)$ the set of bounded linear operators mapping $H$ into $H$. For
$T\in B(H)$, denote the kernel of $T$ and the range of $T$ by Ker$T$
and Ran$T$ respectively. Denote by $\sigma(T), \sigma_e(T),
\sigma_{lre}(T)$ and $\sigma_w(T)$ the spectrum, the essential
spectrum, the wolf spectrum and the weyl spectrum of $T$
respectively. For $\lambda\in \rho_{s-F}(T):= \mathbb{C}\backslash
\sigma_{lre}(T),$ ind$(\lambda-T)= $
dimKer$(\lambda-T)-$dimKer$(\lambda-T)^*, \mbox{min
ind}(\lambda-T)=\mbox{min}\{\mbox{dimKer}(\lambda-T),
\mbox{dimKer}(\lambda-T)^*\}$. Denote
$\rho_{s-F}^{n}(T)=\{\lambda\in\rho_{s-F}(T);~\text{ind}(\lambda-T)=n\}$,
where $-\infty \leq n\leq \infty$,
$\rho_{s-F}^{+}(T)=\{\lambda\in\rho_{s-F}(T);~\text{ind}(\lambda-T)>0\}$
and
$\rho_{s-F}^{-}(T)=\{\lambda\in\rho_{s-F}(T);~\text{ind}(\lambda-T)<0\}$.
According to \cite{HER} corollary 1.14 we know that the funtion
$\lambda\rightarrow \text{min ind}(\lambda-T)$ is constant on every
component of $\rho_{s-F}(T)$ except for an at most denumerable
subset $\rho_{s-F}^s(T)$ without limit points in $\rho_{s-F}(T)$.
Furthermore, if $\mu\in\rho_{s-F}^s(T)$ and $\lambda$ is a point of
$\rho_{s-F}(T)$ in the same component as $\mu$ but $\lambda$ is not
in $\rho_{s-F}^s(T)$, then $\text{min ind}(\mu-T)>\text{min
ind}(\lambda-T)$. $\rho_{s-F}^s(T)$ is the set of $singular ~points$
of the semi-Fredholm domain $\rho_{s-F}(T)$ of $T$;
$\rho_{s-F}^r(T)=\rho_{s-F}(T)\backslash\rho_{s-F}^s(T)$ is the set
of $regular~points$. Denote by $\sigma_0(T)$ the set of isolated
points of $\sigma(T)\backslash \sigma_e(T).$ Denote by
$\overline{E}$ and $E^0$, the closure and the interior of set $E$
respectively. In addition, denote by $LY(H),~DC(H),~UN(H)$ the set
of all Li-Yorke chaotic operators, the set of all distributionally
chaotic operators and the set of all norm-unimodal operators on $H$
respectively.

Now we are in a position to state the main results of this article.
In section 2, the closures and interiors of the sets of all
distributionally chaotic operators or Li-Yorke chaotic operators are
considered. Though distributionally chaotic operators require more
complicated statistical dependence between orbits than Li-Yorke
chaotic operators, we have

{\bf I.} $\overline{DC(H)}=\overline{LY(H)}=\{T\in
B(H);~\partial\mathbb{D}\cap \sigma_{lre}(T)\neq \emptyset \}
\cup\{T\in B(H);~\partial\mathbb{D}\subseteq\rho_{s-F}(T)
~\text{and}~ {\rm dimKer}(\lambda-T)>0,~\forall \lambda
\in\partial\mathbb{D}\}$ (Theorem \ref{5}).

{\bf II.} $DC(H)^0=LY(H)^0=\{T\in B(H),~\exists~ \lambda~ \in~
\partial\mathbb{D}~s.t.~{\rm ind}(\lambda-T)>0\}$ (Theorem \ref{6}).

From the above two results, one can see distributionally chaotic
operators and Li-Yorke chaotic operators are very similar. The
closure of $DC(H)^0$ (i.e. the closure of $LY(H)^0$) is also
considered.

{\bf III.} $\overline{DC(H)^0}=\overline{LY(H)^0}=\{T\in
B(H);~\partial\mathbb{D}\nsubseteq
\rho_{s-F}^{(0)}(T)\cup\rho_{s-F}^{(-)}(T)\}$. Moreover,
$\overline{DC(H)\backslash DC(H)^0}=\overline{LY(H)\backslash
LY(H)^0}=\{T\in B(H);
\partial\mathbb{D}\subseteq\rho_{s-F}^{(0)}(T)\cup\rho_{s-F}^{(-)}(T)$
and ${\rm
dimKer}(\lambda-T)>0,~\forall\lambda\in\partial\mathbb{D}\}
\cup\{T\in B(H);
\partial\mathbb{D}\cap\sigma_{lre}(T)\neq\emptyset~and~\rho_{s-F}^{(+)}(T)\cap\partial\mathbb{D}=\emptyset\}$. (Theorem \ref{7}).

In section 3, we get the relation between hypercyclic operators and
distributionally chaotic operators. In detail, the set of all
hypercyclic operators belongs to the closure of $DC(H)^0$. The
relation between norm-unimodal operators and distributionally
chaotic operators is also obtained.

{\bf IV.} $\overline{UN(H)}=\overline{DC(H)}=\overline{LY(H)}$,
$DC(H)^0=LY(H)^0\subseteq UN(H)$, and \\
$\overline{UN(H)\backslash DC(H)^0}=\overline{DC(H)\backslash
DC(H)^0}=\overline{LY(H)\backslash LY(H)^0}$. (Theorem \ref{9}).

It follows from this result that, the norm-unimodal operators are
very large in the class of distributionally chaotic operators.
Moreover, it is useful for people to prove that an operator is
distributionally chaotic as the criterion of hypercyclic operators
given by Kitai \cite{Kitai} and refined by Grosse-Erdmann and
Shapiro, et.al. \cite{Grosse}.

In section 4, we consider the connectedness of the sets considered
above.

{\bf V.} $DC(H)^0,~\overline{DC(H)^0}$, $\overline{DC(H)}$ and
$\overline{DC(H)\backslash DC(H)^0}$ ( i.e. $LY(H)^0$,
$\overline{LY(H)^0}$, $\overline{LY(H)}$ and
$\overline{LY(H)\backslash LY(H)^0}$ ) are all arcwise connected.
(Theorem \ref{10}).

\section{Closures and interiors of the sets of all distribitionally chaotic operators or Li-Yorke chaotic operators}

Firstly, we need some lemmas which will be used in the proof theorem
\ref{5}. The definition given by Cowen and Douglas \cite{Cowen} is
well known as follows.
\begin{definition}
For $\Omega$ a connected open subset of $\mathbb{C}$ and $n$ a
positive integer, let $B_{n}(\Omega)$ denotes the operators $T$ in
$B(H)$ which satisfy:

(1) $\Omega \subseteq \sigma(T)$;

(2) ${\rm ran}(T-\omega)=H \ for \ \omega \ in \ \Omega$;

(3) $\bigvee _{\omega\in \Omega}{\rm ker}(T-\omega)=H$; and

(4) ${\rm dimker}(T-\omega)=n$ for $\omega$ in $\Omega$.
\end{definition}

One often calls the operator $T$ in $B_{n}(\Omega)$ Cowen-Douglas
operator. Denote by $\mathbb{D}$ and $\partial\mathbb{D}$ the unit
open disk and its boundary. Then we have the following theorem.
\begin{theorem}\cite{Hbz}\label{C-D-Operator}
Let $T\in B_{n}(\Omega)$. If $\Omega\cap
\partial\mathbb{D}\neq \phi$, then $T$ is norm-unimodal.
Consequently, $T$ is distributionally chaotic.
\end{theorem}

\begin{remark}
In fact, this result can be extended to $n=\infty$.
\end{remark}

\begin{lemma}\label{1}
Let $T\in B(H)$. Then the following statements are equivalent.

$(1)~ T$ is not Li-Yorke chaotic.

$(2)~ \liminf\limits_{n\rightarrow\infty}||T^{n}(x)||=0$ implies
$\lim\limits_{n\rightarrow\infty}||T^{n}(x)||=0$.
\end{lemma}
The proof is easy and left to the reader.

\begin{lemma}\label{2}
Let $T\in B(H)$,\ $\sigma(T)\cap\partial\mathbb{D}=\emptyset$.\ Then
$\liminf\limits_{n\rightarrow\infty}||T^{n}(x)||=0$ implies
$\lim\limits_{n\rightarrow\infty}||T^{n}(x)||=0$. Moreover, $T$ is
neither Li-Yorke chaotic nor distributionally chaotic.
\end{lemma}
\begin{proof}
Since $\sigma(T)\cap\partial\mathbb{D}=\emptyset$, then according to
Riesz's decomposition theorem
$$T=\begin{matrix}\begin{bmatrix}
T_{1}\\
&T_{2}\\
\end{bmatrix}&
\begin{matrix}
H_1\\
 H_2\end{matrix}\end{matrix},$$
where $\sigma(T_{1})=\sigma(T)\cap \mathbb{D}~\text{and}~
\sigma(T_{2})=\sigma(T)-\sigma(T_{1})$. Furthermore,
$$T=\begin{matrix}\begin{bmatrix}
T_{1}&*\\
&\widetilde{T_{2}}\\
\end{bmatrix}&
\begin{matrix}
H_1\\
  H_1^{\perp}\end{matrix}\end{matrix},$$ where $\widetilde{T_2}\sim
  T_{2}$ and then $\sigma(\widetilde{T_2})=\sigma(T_2)=\sigma(T)-\sigma(T_{1})$.

By spectral mapping theorem and spectral radius formula,
$$r_{1}(\widetilde{T_2})^{-1}=r(\widetilde{T_2}^{-1})=\lim\limits_{n\rightarrow\infty}\|\widetilde{T_2}^{-n}\|^{\frac{1}{n}},~(\text{where}~r_1(\cdot)=\text{inf}\{|\lambda|;~\lambda\in\sigma(\cdot)\}).$$
Notice $r_{1}(\widetilde{T_2})\geq \delta > 1$, then there is
$\epsilon
> 0$ such that $r_{1}(\widetilde{T_2})^{-1}+\epsilon < 1$. Hence we have $M \in \mathbb{N}$ such that for any $n\geq
M$,\
$$
\frac{1}{\|\widetilde{T_2}^{-n}\|} \geq
(\frac{1}{r_{1}(\widetilde{T_2})^{-1}+\epsilon})^{n}.
$$
Furthermore, for any $y\in{H_1}^{\bot}$,
$$
||\widetilde{T_2}^{n}(y)||\geq
\frac{1}{||\widetilde{T_2}^{-n}||}||y|| \geq
(\frac{1}{r_{1}(\widetilde{T_2})^{-1}+\epsilon})^{n}\|y\|\geq
\|y\|,~\text{when}~n\geq M.
$$

Let $\liminf\limits_{n\rightarrow\infty}||T^{n}(x)||=0$ and
$x=x_1\oplus x_2,~x_1\in H_1,~x_2\in{H_1}^{\bot}$, then one can
easily obtain $x_2=0$ and then $T^{n}(x)=T_{1}^{n}(x_{1})$. On the
other hand $r(T_{1}) < 1$, so there exist $0\leq\rho<1$ and $N\in
\mathbb{N}$ such that for any $n\geq N$, $\|T_1^n(x_1)\|\leq \rho^n
\|x_1\|$. Therefore,
$\lim\limits_{n\rightarrow\infty}||T^{n}(x)||=\lim\limits_{n\rightarrow\infty}||{T_1}^{n}(x_1)||=0$.
\end{proof}

Following from Kitai's result in \cite{Kitai}, no finite dimensional
Hilbert space supports a hypercyclic operator. In fact, we have

\begin{lemma}\label{4}
Let $0<n<\infty$ be an integer and $T\in B(\mathbb{C}^n)$. Then
$\liminf\limits_{m\rightarrow\infty}||T^{m}(x)||=0$ implies
$\lim\limits_{m\rightarrow\infty}||T^{m}(x)||=0$. Moreover, $T$ is
neither Li-Yorke chaotic nor distributionally chaotic.
\end{lemma}

From \cite{HBZ}, one can see

\begin{lemma}\label{per}
For any $\epsilon>0$, there is a compact operator $K_\epsilon\in
B(H)$ such that $\|K_{\epsilon}\|<\epsilon$ and $I+K_{\epsilon}$ is
distributionally chaotic.
\end{lemma}

Now we will give a description of the closures for the sets of
distributionally chaotic operators or Li-Yorke chaotic operators.

\begin{theorem}\label{5}
Let $E_1=\{T\in B(H);~\partial\mathbb{D}\cap \sigma_{lre}(T)\neq
\emptyset \}$ and $E_2=\{T\in
B(H);~\partial\mathbb{D}\subseteq\rho_{s-F}(T)~and~{\rm
dimKer}(\lambda-T)>0,~\forall \lambda \in\partial\mathbb{D}\}$. Then
$\overline{DC(H)}=\overline{LY(H)}=E_1\cup E_2.$
\end{theorem}
\begin{proof}
Clearly, $\overline{DC(H)}\subseteq \overline{LY(H)}$. So it
suffices to show $E_1\cup E_2\subseteq\overline{DC(H)}$ and
$\overline{LY(H)}\subseteq E_1\cup E_2$.

First step, $E_1\cup E_2\subseteq\overline{DC(H)}$. We will show for
any $T\in E_1\cup E_2$ and $\epsilon>0$, there exists an operator
$C$ such that $||C||<\epsilon$ and $T+C \in DC(H)$. In fact, one can
obtain a compact operator $K$ such that $||K||<\epsilon$ and $T+K
\in DC(H)$.

If $T\in E_1$, then choose a $\lambda_0 \in
\partial\mathbb{D}\cap \sigma_{lre}(T)$. By AFV theorem
there exists a compact operator $K_1$ such that $||K_1||<\epsilon/2$
and
$$T+{K_1}=\begin{matrix}\begin{bmatrix}
\lambda_0I&*\\
&*\\
\end{bmatrix}&
\begin{matrix}
 H_0\\
  {H_0}^{\bot}\end{matrix}\end{matrix},$$
  where $\text{dim}H_0=\infty$.

Following lemma \ref{per} there exists a compact operator $K_2$ such
that $||K_2||<\epsilon/2$ and $\lambda_0I+K_2$ is distributionally
chaotic. Let $$\widetilde{K_2}=\begin{matrix}\begin{bmatrix}
K_2\\
&0\\
\end{bmatrix}&
\begin{matrix}
 H_0\\
  {H_0}^{\bot}\end{matrix}\end{matrix}.$$
Then $T+K_1+\widetilde{K_2}\in DC(H)$, where $K_1+\widetilde{K_2}$
is a compact operator and $||K_1+\widetilde{K_2}||<\epsilon$.

If $T\in E_2$, define
$$H_r=\underset{\lambda\in\rho_{s-F}^{r}(T)\cap\Delta}{\bigvee}\text{Ker}(\lambda-T),$$
where $\Delta $ is the component of semi-Fredholm domain
$\rho_{s-F}(T)$ which contains $\partial\mathbb{D}$. Then
$\text{dim} H_r=\infty$ and
$$T=\begin{matrix}\begin{bmatrix}
T_r&*\\
&*\\
\end{bmatrix}&
\begin{matrix}
 H_r\\
  {H_r}^{\bot}\end{matrix}\end{matrix}.$$
Since
\begin{eqnarray*}
&(1)& \rho_{s-F}^{r}(T)\cap\Delta\subseteq\sigma(T_r),\\[2mm]
&(2)&
\text{dimKer}(\mu-T_r)=\text{dimKer}(\mu-T)=n,~\forall\mu\in{\rho_{s-F}^{r}(T)\cap\Delta},
\end{eqnarray*}

where $n\in\mathbb{N}^+\cup\{\infty\},$
\begin{eqnarray*} &(3)&
\underset{\mu\in\rho_{s-F}^{r}(T)\cap\Delta}{\bigvee}Ker(\mu-T_r)=\underset{\mu\in\rho_{s-F}^{r}(T)\cap\Delta}{\bigvee}Ker(\mu-T)=H_r,\\[2mm]
&(4)& \text{Ran}(\mu-T_r)=H_r,~\mu\in\rho_{s-F}^{r}(T)\cap\Delta,
\end{eqnarray*}
we know $T_r\in B_{n}(\rho_{s-F}^{r}(T)\cap\Delta)$. Notice
$\rho_{s-F}^{r}(T)\cap\Delta\cap\partial\mathbb{D}\neq\emptyset$, by
theorem \ref{C-D-Operator} $T_r$ is norm-unimodal and hence
distributionally chaotic. So is $T$. The first step is complete.

Second step, $\overline{LY(H)}\subseteq E_1\cup E_2$. Notice
$$\{E_1\cup E_2\}^c=\{T\in
B(H);~\partial\mathbb{D}\subseteq\rho_{s-F}(T)~\text{and}~\exists\lambda\in\partial\mathbb{D}~\text{s.t.}~{\rm
dimKer}(\lambda-T)=0\},$$ by Fredholm theory $\{E_1\cup E_2\}^c$ is
open. Since $\{\overline{LY(H)}\}^c=\{LY(H)^c\}^0$, it suffices to
prove $\{E_1\cup E_2\}^c\subseteq LY(H)^c$.

Let $T\in \{E_1\cup E_2\}^c$, define
$$H_l=\underset{\lambda\in\rho_{s-F}^{r}(T)\cap\Phi}{\bigvee}\text{Ker}(\lambda-T)^*,$$
where $\Phi$ is the component of semi-Fredholm domain
$\rho_{s-F}(T)$ which contains $\partial\mathbb{D}$. Then
$$T=\begin{matrix}\begin{bmatrix}
T_0&*\\
&T_l\\
\end{bmatrix}&
\begin{matrix}
{H_l}^{\bot} \\
  H_l\end{matrix}\end{matrix},~(H_l ~\text{maybe}~ \{0\}!).$$

Claim 1: $\rho_{s-F}^{r}(T)\cap\Phi\subseteq\rho(T_0)$.

Let $\mu\in\rho_{s-F}^{r}(T)\cap\Phi$. Since $\lambda\rightarrow$
min ind$(\lambda-T)$ is constant on the semi-Fredholm domain
$\rho^{r}_{s-F}(T)$ and $\exists\lambda_0\in\partial\mathbb{D}$ s.t.
dimKer$(\lambda_0-T)=0$, we have Ker$(\mu-T)=\{0\}$. Hence
Ker$(\mu-T_0)=\{0\}$. Notice Ker$(\mu-T)^*=\text{Ker}(\mu-T_l)^*$
and $(\mu-T)^*(H_l)=H_l$, then Ker$(\mu-T_0)^*=\{0\}$. Therefore,
$\mu-T_0$ is invertible.

Claim 2:
$\sigma_0(T_0)\cap\Phi=\sigma(T_0)\cap\Phi=\rho_{s-F}^{s}(T)\cap\Phi$.

From claim 1,
$\sigma_0(T_0)\cap\Phi\subseteq\sigma(T_0)\cap\Phi\subseteq\rho_{s-F}^{s}(T)\cap\Phi$.
Let $\lambda\in\rho_{s-F}^{s}(T)\cap\Phi$. If $\lambda-T_0$ is
invertible, then $\lambda-T_l$ is a semi-Fredholm operator and min
ind$(\lambda-T_l)$=min ind$(\lambda-T)$. Since
$(\lambda-T_l)^*(H_l)=H_l$, dimKer$(\lambda-T_l)=0$ and hence min
ind$(\lambda-T)$=min ind$(\lambda-T_l)$=0. It is contradict to
$\lambda\in\rho_{s-F}^s(T)$. Therefore, $\lambda-T_0$ is not
invertible and
$\rho_{s-F}^s(T)\cap\Phi\subseteq\sigma(T_0)\cap\Phi$. Because
$\lambda-T$ is a left semi-Fredholm operator,
$(\lambda-T)(H_l^{\bot})$ is closed. Therefore $\lambda-T_0$ is a
semi-Fredholm operator. Again notice claim 1 we obtain
$\rho_{s-F}^s(T)\cap\Phi\subseteq\sigma_0(T_0)\cap\Phi$.

Since the only limit points of $\rho_{s-F}^s(T)$ belong to
$\partial[\rho_{s-F}(T)]$, let
$\sigma_0(T_0)\cap\Phi\cap\partial\mathbb{D}=\{\mu_i\}_{i=1}^{m},~m<\infty$.
By Riesz's decomposition theorem and Rosenblum-Davis-Rosenthal
corollary \cite{HER},
$$T_0=\begin{matrix}\begin{bmatrix}
T_{00}\\
&T_{01}
\end{bmatrix}&
\begin{matrix}
H_{00}\\
H_{01}
\end{matrix}\end{matrix}=
\begin{matrix}\begin{bmatrix}
T_{00}&*\\
&\widetilde{T_{01}}
\end{bmatrix}&
\begin{matrix}
H_{00}\\
{H_l}^{\bot}\ominus H_{00}
\end{matrix}\end{matrix}\sim
\begin{matrix}\begin{bmatrix}
T_{00}\\
&\widetilde{T_{01}}
\end{bmatrix}&
\begin{matrix}
H_{00}\\
{H_l}^{\bot}\ominus H_{00}
\end{matrix}\end{matrix},$$ where
$\sigma(T_{00})=\{\mu_i\}_{i=1}^{m},~m<\infty$,
$\sigma(T_{01})\cap\partial\mathbb{D}=\emptyset$, $T_{01}\sim
\widetilde{T_{01}}$ and dim$H_{00}<\infty$. Hence
$$T\sim S:=\begin{matrix}\begin{bmatrix}
T_{00}&0&*\\
&\widetilde{T_{01}}&*\\
&&T_l
\end{bmatrix}&
\begin{matrix}
H_{00}\\
{H_l}^{\bot}\ominus H_{00}\\
H_l
\end{matrix}\end{matrix}.$$
Moreover
$H_l=\underset{\lambda\in\rho_{s-F}^{r}(T)\cap\Phi}{\bigvee}{\rm
Ker}(\lambda-T)^*=\underset{\lambda\in\partial\mathbb{D}\cap\rho_{s-F}^{r}(T)\cap\Phi}{\bigvee}{\rm
Ker}(\lambda-T)^*$, then
$$T_l=\begin{matrix}\begin{bmatrix}
\lambda_1\\
*&\lambda_2\\
*&*&\lambda_3\\
\vdots&\vdots&\vdots&\ddots
\end{bmatrix}&
\begin{matrix}
e_1\\
e_2\\
e_3\\
\vdots
\end{matrix}\end{matrix},$$where $\{\lambda_i\}_{i=1}^{\infty}\subseteq \partial\mathbb{D}\cap\rho_{s-F}^{r}(T)\cap\Phi$
and $\{e_i\}_{i=1}^{\infty}$ is an ONB of $H_l$.

Now we come to end the proof. Since Li-Yorke chaos is invariant
under similar and lemma \ref{1}, it suffices to show
$\liminf\limits_{n\rightarrow\infty}||S^n(x)||=0$ implies
$\lim\limits_{n\rightarrow\infty}||S^n(x)||=0$. Let
$\liminf\limits_{n\rightarrow\infty}||S^n(x)||=0$, then there exist
$\{n_k\}_{k=1}^{\infty}$ such that
$\lim\limits_{n_k\rightarrow\infty}||S^{n_k}(x)||=0$. Notice
$x=x_0\oplus \widetilde{x_0}\oplus x_l,~x_0\in
H_{00},~\widetilde{x_0}\in {H_l}^{\bot}\ominus H_{00},~x_l\in H_l$,
we have $\lim\limits_{n_k\rightarrow\infty}||T_l^{n_k}(x_l)||=0$.
Following the matrix representation of $T_l$, $x_l=0$. Hence,
$$S^{n_k}(x)={\begin{matrix}\begin{bmatrix}
T_{00}\\
&\widetilde{T_{01}}
\end{bmatrix}&
\begin{matrix}
\end{matrix}\end{matrix}}^{n_k}\begin{matrix}\begin{bmatrix}
x_0\\
\widetilde{x_0}
\end{bmatrix}&
\begin{matrix}
\end{matrix}\end{matrix}.$$ So by lemma \ref{4} and lemma
\ref{2} $\lim\limits_{n\rightarrow\infty}||T_{00}^n(x_0)||=0$ and
$\lim\limits_{n\rightarrow\infty}||\widetilde{T_{01}}^n(\widetilde{x_0})||=0$.
That is $\lim\limits_{n\rightarrow\infty}||S^n(x)||=0$. The second
step is complete.
\end{proof}

Theorem \ref{5} also includes the information of the interior for
$DC(H)^c$. Obviously, the operator $T$ satisfying,
$\sigma(T)\cap\partial\mathbb{D}=\emptyset$, is in $\{DC(H)^c\}^0$
(i.e. $\{\overline{DC(H)}\}^c$). There exists an operator $T$ in
$\{DC(H)^c\}^0$, whose spectrum $\sigma(T)$ intersects the unit
circle, i.e. $\sigma(T)\cap\partial\mathbb{D}\neq\emptyset$.
\begin{example}
Let $A\in B(H)$ satisfying
\begin{align*}\left\{
\begin{array}{ll}\hspace{7mm}Ae_i=\frac{1}{2}e_{i+1},\ \hspace{20mm}  i\leq-2,
\\[2mm]
\hspace{7mm}Ae_i=2e_{i+1}, \hspace{22mm} i>-2,\\
\end{array}\right.
\end{align*}
where $\{e_i\}_{i=-\infty}^{\infty}$ is an ONB of $H$. Then $A$ is
in $\{DC(H)^c\}^0$.
\end{example}
\begin{proof}
Through easy compute, one can obtain
$\sigma(A)=\{z\in\mathbb{C};~\frac{1}{2}\leq |z|\leq 2\}$ and
ind$(\lambda-A)=-1$, dimKer$(\lambda-A)=0$ for
$\lambda\in\{z\in\mathbb{C};~\frac{1}{2}< |z|< 2\}$. According to
theorem \ref{5}, $A\in\{\overline{DC(H)}\}^c$ (i.e.
$\{DC(H)^c\}^0$).

We consider the dynamical property of $A$. For any $x$ in $H$,
$x=\Sigma_{i=-\infty}^{\infty}x_ie_i$ and
\begin{eqnarray*}\\
A^{2n+1}(x)=&(&\cdots,\frac{1}{2^{2n+1}}x_{-(2n+2)},\widehat{\frac{1}{2^{2n-1}}x_{-(2n+1)}},
\frac{1}{2^{2n-3}}x_{-2n},\cdots, \\[2mm]
&&\cdots, \frac{1}{2}x_{-(n+2)},2x_{-(n+1)}, 2^3x_{-n}, \cdots),\\[2mm]
A^{2n}(x)=&(&\cdots,\frac{1}{2^{2n}}x_{-2n-1},\widehat{\frac{1}{2^{2n-2}}x_{-2n}},
\frac{1}{2^{2n-4}}x_{-2n+1},\cdots, \\[2mm]
&&\cdots, x_{-(n+1)},2^2x_{-n}, 2^4x_{-n+1}, \cdots),\\
\end{eqnarray*}
where the position under $\wedge$ is the 0 position corresponding to
the ONB $\{e_i\}_{i=-\infty}^{\infty}$. One can easily obtain if
$x\neq0$, then $||A^n(x)||\rightarrow\infty$. Hence $A$ is not
distributionally chaotic.
\end{proof}

Next, we consider the interiors of the sets of all Li-Yorke chaotic
operators or distributionally chaotic operators. Before proving the
result theorem \ref{6}, it is convenient to cite in full length a
result of Apostol and Morrel. Let $\Gamma=\partial\Omega$, where
$\Omega$ is an analytic Cauchy domain, and let $L^2(\Gamma)$ be the
Hilbert space of (equivalent classes of) complex functions on
$\Gamma$ which are square integrable with respect to
(1/2$\pi$)-times the arc-length measure on $\Gamma$; $M(\Gamma)$
will stand for the operator defined as multiplication by $\lambda$
on $L^2(\Gamma)$. The subspace $H^2(\Gamma)$ spanned by the rational
functions with poles outside $\overline{\Omega}$ is invariant under
$M(\Gamma)$. By $M_+(\Gamma)$ and $M_-(\Gamma)$ we shall denote the
restriction of $M(\Gamma)$ to $H^2(\Gamma)$ and its compression to
$L^2(\Gamma)\ominus H^2(\Gamma)$,respectively, i.e.
$$M(\Gamma)=\begin{matrix}\begin{bmatrix}
M_+(\Gamma)&Z\\
&M_-(\Gamma)\end{bmatrix} &
\begin{matrix}
H^2(\Gamma)\\
H^2(\Gamma)^\bot\end{matrix}\end{matrix}.
$$
\begin{definition}\cite{HER}
$S\in B(H)$ is a simple model, if it has the form
$$S=\begin{matrix}\begin{bmatrix}
S_+&*&*\\
&A&*\\
&&S_-\\
\end{bmatrix}&
\begin{matrix}
\end{matrix}\end{matrix},$$
where

(1) $\sigma(S_+),~\sigma(S_-),~\sigma(A)$\ are pairwise disjoint;

(2) A is similar to a normal operator with finite spectrum;

(3) $S_+$ is (either absent or) unitarity equivalent to
$\oplus_{i=1}^{m}M_+(\partial\Omega_i)^{k_i},~1\leq k_i \leq\infty$,
where $\{\partial\Omega_i\}_{i=1}^{m}$ is a finite family of
analytic Cauchy domains with pairwise diajoint closures;

(4) $S_-$ is (either absent or) unitarity equivalent to
$\oplus_{j=1}^{n}M_-(\partial\Phi_j)^{h_j},~1\leq h_j \leq\infty$,
where $\{\partial\Phi_j\}_{j=1}^{n}$ is a finite family of analytic
Cauchy domains with pairwise diajoint closures.
\end{definition}
\begin{theorem}\cite{HER}\label{Apostol-Morrel}
The simple models are dense in $B(H)$. More precisely: Given $T\in
B(H)$ and $\epsilon>0$ there exists a simple model $S$ such that

$(1)~ \sigma(S_+)\subseteq \rho_{s-F}^-(T)\subseteq
\sigma(S_+)_{\epsilon},~\sigma(S_-)\subseteq\rho_{s-F}^+(T)\subseteq
 \sigma(S_-)_{\epsilon},$  and  $\sigma(A)\subseteq\sigma(T)_{\epsilon}.$

$(2)~ {\rm ind}(\lambda-S)={\rm ind}(\lambda-T), \ for \ each \
\lambda\in \rho_{s-F}^-(S_+)\cup\rho_{s-F}^+(S_-).$

$(3)~ ||T-S||<\epsilon.$
\end{theorem}

Additionally, we give a lemma which appeared in another article
\cite{HBZ}. But for convenience to read this article, we also give
the details of the proof.
\begin{lemma}
Let $N\in B(H)$ be a normal operator. Then
$\liminf\limits_{n\rightarrow\infty}||N^{n}(x)||=0$ implies
$\lim\limits_{n\rightarrow\infty}||N^{n}(x)||=0$. Moreover, $N$ is
neither Li-Yorke chaotic nor distributionally chaotic.
\end{lemma}
\begin{proof}
Since $N$ is a normal operator, then there exist a locally compact
space $X$, a finite positive regular Borel measure $\mu$ and a Borel
function $\eta \in L^{\infty}(X, \mu)$ such that $N$ and $M_{\eta}$
are unitarily equivalent. $M_{\eta}$ is multiplication by $\eta$ on
$L^{2}(X, \mu)$. Let
$\liminf\limits_{n\rightarrow\infty}\|M_{\eta}^n(f)\|=0$ and
\begin{eqnarray*}
&&\Delta_{1}=\{z\in X ; |\eta(z)|\geq 1 \} ,\\
&&\Delta_{2}=\{z\in X ; |\eta(z)|< 1 \},\\
&&\Delta_{3}=\{z\in X ; f(z)=0 \ \ a.e. \ [\mu] \},\\
&&\Delta_{4}=\{z\in X ; f(z)\neq 0 \ \ a.e. \ [\mu] \}.
\end{eqnarray*}
Then there exists a sequence of positive integers
$\{n_{k}\}_{k=1}^{\infty} $ such that
$\lim\limits_{n_k\rightarrow\infty}\|M_{\eta}^{n_k}(f)\|=0$ and
\begin{eqnarray*}
\|M_{\eta}^{n_k}(f)\|^{2}&=&\int_{X}|\eta^{n_k}f|^{2}d\mu \\
&=& \int_{\Delta_{1}\cap \Delta_{4}}|\eta^{n_k}f|^{2}d\mu +
\int_{\Delta_{2}\cap \Delta_{4}}|\eta^{n_k}f|^{2}d\mu \\
&\geq& \int_{\Delta_{1}\cap \Delta_{4}}|f|^{2}d\mu +
\int_{\Delta_{2}\cap \Delta_{4}}|\eta^{n_k}f|^{2}d\mu.
\end{eqnarray*}
Consequently $\mu(\Delta_{1}\cap \Delta_{4})=0$. For any $n\in
\mathbb{N}$, there exists a positive integer $k$ such that $n_k \leq
n < n_{k+1}$. Therefore,
\begin{eqnarray*}
||M_{\eta}^n(f)||^2 &=& \int_{\Delta_{2}\cap \Delta_{4}}|\eta^n
f|^{2}d\mu \\
&=& \int_{\Delta_{2}\cap \Delta_{4}}|\eta^{n_k}
f|^{2}|\eta^{n-n_{k}}|^2d\mu \\
&\leq& \int_{\Delta_{2}\cap \Delta_{4}}|\eta^{n_k} f|^{2}d\mu \\
&=& \|M_{\eta}^{n_k}(f)\|^{2},
\end{eqnarray*}
and hence $\lim\limits_{n\rightarrow\infty}\|M_{\eta}^n(f)\|=0$.
Notice the property which we considered is invariant under unitarily
equivalence, we obtain the result.
\end{proof}
\begin{corollary}\label{3}
Let $T\in B(H)$ be a subnormal operator. Then
$\liminf\limits_{n\rightarrow\infty}||T^{n}(x)||=0$ implies
$\lim\limits_{n\rightarrow\infty}||T^{n}(x)||=0$. Moreover, $T$ is
neither Li-Yorke chaotic nor distributionally chaotic.
\end{corollary}

\begin{theorem}\label{6}
Let $F=\{T\in B(H),~\exists~ \lambda~ \in~
\partial\mathbb{D}~s.t.~{\rm ind}(\lambda-T)>0\}$. Then $DC(H)^0=LY(H)^0=F$.
\end{theorem}
\begin{proof}
Obviously, $DC(H)^0\subseteq LY(H)^0$. So we only need to show
$F\subseteq DC(H)^0$ and $LY(H)^0\subseteq F$.

First step, $F\subseteq DC(H)^0$. By Fredholm theory, $F$ is open.
Consequently, it suffices to prove $F\subseteq DC(H)$.

Let $T\in F$. Define
$$H_r=\underset{\mu\in\rho_{s-F}^{r}(T)\cap\Delta}{\bigvee}{\rm Ker}(\mu-T),$$
where $\Delta$ is the component of $\rho_{s-F}^{(+)}(T)$ which
contains a point in $\partial\mathbb{D}$. Then dim$H_r = \infty$ and
$$T=\begin{matrix}\begin{bmatrix}
T_r&*\\
&*\\
\end{bmatrix}&
\begin{matrix}
H_r\\
 {H_r}^\bot\end{matrix}\end{matrix}.$$
Similar to theorem \ref{5}, we have $T_r\in
B_n(\rho_{s-F}^{r}(T)\cap\Delta)$, moreover $T$ is norm-unimodal and
distributionally chaotic. The first step is complete.

Next step, $LY(H)^0\subseteq F$. Because
$\{LY(H)^0\}^c=\overline{LY(H)^c}$, we only need to show for any
$T\in F^c$ and $\epsilon> 0$, there exists $C\in B(H)$ such that
$||C||<\epsilon$ and $T+C$ is not Li-Yorke chaotic.

Let $T\in F^c$ and $\epsilon>0$. According to Theorem
\ref{Apostol-Morrel}, there exists a simple model
$$S=\begin{matrix}\begin{bmatrix}
S_+&*&*\\
&A&*\\
&&S_-\\
\end{bmatrix}&
\begin{matrix}
\end{matrix}\end{matrix}$$ such that
$ \sigma(S_-)\subseteq \rho_{s-F}^+(T)\subseteq
\sigma(S_-)_{\epsilon}$ and $||T-S||<\epsilon$. Since
$\rho_{s-F}^{(+)}(T)\cap\partial\mathbb{D}=\emptyset$, we can ensure
$\sigma(S_-)\cap\partial\mathbb{D}=\emptyset$ according to the proof
of theorem \ref{Apostol-Morrel}. So it suffices to prove $S$ is
impossible to be Li-Yorke chaotic.

Notice $\sigma(S_+),~\sigma(S_-),~\sigma(A)$ are pairwise disjoint,
according to Rosenblum-Davis-Rosenthal corollary \cite{HER},
$$S\sim\begin{matrix}\begin{bmatrix}
S_+\\
&A\\
&&S_-\\
\end{bmatrix}&
\begin{matrix}
\end{matrix}\end{matrix}.$$
Because Li-Yorke chaotic is invariant under similar and $A$ is
similar to a normal operator $N$ with finite spectrum, we directly
let $S= S_+\oplus N\oplus S_-$.

If $\liminf\limits_{n\rightarrow\infty}||S^n(x)||=0$, since
 $x=x_+\oplus x_0\oplus
x_-$ corresponding to the space decomposition, following lamma
\ref{2} and corollary \ref{3} we have
$$\lim\limits_{n\rightarrow\infty}||{S_-}^n(x_-)||=0 \ \text{and} \
\lim\limits_{n\rightarrow\infty}||{\begin{matrix}\begin{bmatrix}
S_+\\
&N\\
\end{bmatrix}&
\begin{matrix}
\end{matrix}\end{matrix}}^n\begin{matrix}\begin{bmatrix}
x_+\\
x_0\\
\end{bmatrix}&
\begin{matrix}
\end{matrix}\end{matrix}||=0.$$ Hence $\lim\limits_{n\rightarrow\infty}||S^n(x)||=0$. $S$ is not Li-Yorke
chaotic. The second step is complete.
\end{proof}

We give an example which is distributionally chaotic but not in
$DC(H)^0$.
\begin{example}
Let $A\in B(H)$ satisfying
\begin{align*}\left\{
\begin{array}{ll}\hspace{7mm}Ae_i=2e_{i-1},\ \hspace{20mm}  i\neq0,
\\[2mm]
\hspace{7mm}Ae_0=0,
\end{array}\right.
\end{align*}
where $\{e_i\}_{i=-\infty}^{\infty}$ is an ONB of $H$. Then $A$ is
distributionally chaotic but not in $DC(H)^0$.
\end{example}
\begin{proof}
Since $H_0=\vee_{i=0}^{\infty}\{e_i\}$ is an invariant space of $A$
and $A|_{H_0}=2B$ is distributionally chaotic, where $B$ is backward
unilateral shift, then $A$ is distributionally chaotic. One can
easily obtain ind$(\lambda-A)=0$ for $|\lambda|<2$, so $A$ is not in
$DC(H)^0$.

We prove it directly. For any $\epsilon>0$, let $K\in B(H)$
satisfying $Ke_0=\epsilon e_{-1},~Ke_i=0,~i\neq0$. Then $K$ is
compact and $||K||=\epsilon$. Since
$\sigma(A+C)=\{z\in\mathbb{C};~|z|=2\}$, $A+K$ is not
distributionally chaotic. Hence $A$ is not in $DC(H)^0$.
\end{proof}

Unfortunately, the closure of $DC(H)^0$ (i.e. the closure of
$LY(H)^0$) is not equal to the closure of $DC(H)$ (i.e. the closure
of $LY(H)$). It means there exists a class of distributionally
chaotic operators (Li-Yorke chaotic operators) which are more
complicated. We give these descriptions.
\begin{theorem}\label{7}
Let $G_0=\{T\in B(H);~\partial\mathbb{D}\nsubseteq
\rho_{s-F}^{(0)}(T)\cup\rho_{s-F}^{(-)}(T)\}$, $G_1=\{T\in B(H);
\partial\mathbb{D}\subseteq\rho_{s-F}^{(0)}(T)\cup\rho_{s-F}^{(-)}(T)
~and~{\rm
dimKer}(\lambda-T)>0,~\forall\lambda\in\partial\mathbb{D}\}$ and
$G_2=\{T\in B(H);
\partial\mathbb{D}\cap\sigma_{lre}(T)\neq\emptyset~and~\rho_{s-F}^{(+)}(T)\cap\partial\mathbb{D}=\emptyset\}$. Then
$\overline{DC(H)^0}=\overline{LY(H)^0}=G_0$ and
$\overline{DC(H)\backslash DC(H)^0}=\overline{LY(H)\backslash
LY(H)^0}= G_1\cup G_2$.
\end{theorem}
\begin{proof}
First, we prove $\overline{DC(H)^0}=\overline{LY(H)^0}=G_0$.
Clearly, $DC(H)^0=LY(H)^0=F\subseteq G_0$, where $F$ is denoted in
theorem \ref{6}. By Fredholm theory, $G_0$ is closed. Hence
$\overline{DC(H)^0}=\overline{LY(H)^0}\subseteq G_0$. So we only
need to show for any $T\in G_0$ and $\epsilon>0$, there exists $C\in
B(H)$ such that $||C||<\epsilon$ and $T+C\in DC(H)^0$ (i.e.
$LY(H)^0$).

Let $T\in G_0$ and $\epsilon>0$. Then
\begin{eqnarray*}
&&(1) \exists\lambda\in\partial\mathbb{D} ~\text{s.t. ind}(\lambda-T)>0, ~~\text{or}\\
&&(2) \rho_{s-F}^{(+)}(T)\cap\partial\mathbb{D}=\emptyset
,~\text{but}~\exists\lambda\in\partial\mathbb{D} ~\text{s.t. ind}(\lambda-T)=0,~~ \text{or}\\
&&(3)
[\rho_{s-F}^{(+)}(T)\cup\rho_{s-F}^{(0)}(T)]\cap\partial\mathbb{D}=\emptyset,~\text{but}~\sigma_{lre}(T)\cap\partial\mathbb{D}\neq\emptyset.
\end{eqnarray*}

Case (1) is obvious.

Case (2). First it implies
$\sigma_{lre}(T)\cap\partial\mathbb{D}\neq\emptyset$. Then choose a
$\lambda_0\in\sigma_{lre}(T)\cap\partial[\rho_{s-F}^{(0)}(T)\cap\partial\mathbb{D}]$.
According to AFV theorem, there exists a compact operator $K_1$ such
that $||K_1||<\epsilon/2$ and
$$T+K_1=\begin{matrix}\begin{bmatrix}
\lambda_0I&*\\
&A\\
\end{bmatrix}&
\begin{matrix}
H_0\\
 {H_0}^{\bot}\end{matrix}\end{matrix},$$ where $\sigma(A)=\sigma(T),~\sigma_{lre}(A)=\sigma_{lre}(T)$ and $\text{ind}(\lambda-A)=\text{ind}(\lambda-T)$, for $\lambda\in\rho_{s-F}(T)$.
 Let $$B_\epsilon=\begin{matrix}\begin{bmatrix}
0&\epsilon/2\\
&0&\epsilon/2\\
&&\ddots&\ddots
\end{bmatrix}&
\begin{matrix}
e_1\\
 e_2\\
  \vdots\end{matrix}\end{matrix},$$ where $\{e_i\}_{i=1}^{\infty}$ is an ONB of $H_0$ and $$K_2=\begin{matrix}\begin{bmatrix}
B_\epsilon\\
&0\\
\end{bmatrix}&
\begin{matrix}
H_0\\
{H_0}^{\bot}\end{matrix}\end{matrix}.$$ Obviously,
$||K_1+K_2||<\epsilon$ and there exists
$\lambda\in\partial\mathbb{D}$ such that
ind$(T+K_1+K_2-\lambda)=1>0$. Hence $T+K_1+K_2\in DC(H)^0$.

Case (3). By theorem \ref{Apostol-Morrel} there exists $C_1$ such
that $||C_1||<\epsilon/2$ and
$$T+C_1=\begin{matrix}\begin{bmatrix}
S_+&*&*\\
&A&*\\
&&S_-\\
\end{bmatrix}&
\begin{matrix}
\end{matrix}\end{matrix},$$ where $S_+$ is either absent or
unitarity equivalent to a subnormal operator and
$\partial\mathbb{D}\backslash\sigma(S_+)$ contains a small arc in
$\partial\mathbb{D}$, $A$ is similar to a normal operator with
finite spectrum and
$\sigma_{lre}(A)\cap\partial\mathbb{D}\neq\emptyset$, $S_-$ is
either absent or unitarity equivalent to the adjoint of a subnormal
operator and $\sigma(S_-)\cap\partial\mathbb{D}=\emptyset$;
$\sigma(S_+),~\sigma(A),~\sigma(S_-)$ are pairwise disjoint.
Furthermore,
$$\sigma_{lre}(T+C_1)=\sigma_{lre}(S_+)\cup\sigma_{lre}(A)\cup\sigma_{lre}(S_-)~\text{and}~\rho(T+C_1)=\rho(S_+)\cap\rho(A)\cap\rho(S_-).$$
Hence $$\sigma_{lre}(T+C_1)\cap\partial\mathbb{D}\neq\emptyset \
\text{and} \ \rho(T+C_1)\cap\partial\mathbb{D}\neq\emptyset.$$ Then
we can obtain $C_2$ through the technology of case (2) such that
$||C_2||<\epsilon/2$ and ind$(T+C_1+C_2-\lambda_1)>0$ (where
$\lambda_1\in\partial\mathbb{D}$). Hence $T+C_1+C_2\in DC(H)^0$. The
first equation is complete.

Second, we prove $\overline{DC(H)\backslash
DC(H)^0}=\overline{LY(H)\backslash LY(H)^0}= G_1\cup G_2$. Clearly,
\begin{eqnarray*}\overline{DC(H)}~\backslash~
\overline{DC(H)^0}\subseteq\overline{DC(H)\backslash
DC(H)^0}\subseteq\overline{LY(H)\backslash
LY(H)^0}\subseteq\overline{LY(H)}~\backslash LY(H)^0.
\end{eqnarray*}
Then $G_1\subseteq \overline{DC(H)\backslash
DC(H)^0}\subseteq\overline{LY(H)\backslash LY(H)^0}\subseteq G_1\cup
G_2$. In order to obtain the result, we only need to show $G_2
\subseteq\overline{DC(H)\backslash DC(H)^0}$.

For any $T\in G_2$ and $\epsilon>0$, according to AFV theorem there
exists a compact operator $K_1$ such that $||K_1||<\epsilon/2$ and
$$T+K_1=\begin{matrix}\begin{bmatrix}
\lambda_0I&*\\
&A\\
\end{bmatrix}&
\begin{matrix}
H_0\\
 {H_0}^{\bot}\end{matrix}\end{matrix}$$ where $\lambda_0\in
 \partial\mathbb{D}\cap\sigma_{lre}(T)$. By lamma \ref{per}, we know there exists a compact operator $K_\epsilon$ such that
$ ||K_\epsilon||<\epsilon/2$ and $\lambda_0I+K_\epsilon $ is
distributionally chaotic. Let $$K_2=\begin{matrix}\begin{bmatrix}
K_\epsilon\\
&0\\
\end{bmatrix}&
\begin{matrix}
H_0\\
{H_0}^{\bot}\end{matrix}\end{matrix}.$$ So $||K_1+K_2||<\epsilon$
and $T+K_1+K_2$ is distributionally chaotic. Notice
$$\rho_{s-F}^{(+)}(T+K_1+K_2)\cap\partial\mathbb{D}=\rho_{s-F}^{(+)}(T)\cap\partial\mathbb{D}=\emptyset,$$
we know $T+K_1+K_2\in \{DC(H)^0\}^c$. The second equation is
complete.
\end{proof}

\section{Some other results}
In this section, we consider the relation between hypercyclic
operators and distributionally chaotic operators, and the closure of
the set of all norm-unimodal operators.

Recall the definition of  chaos given by Devaney \cite{Devaney} as
follows.
\begin{definition}
Suppose that $f: X\rightarrow X$ is a continuous function on a
complete separable metric space $X$, then $f$ is Devaney's chaotic
if:

(a) the periodic points for $f$ are dense in $X$,

(b) $f$ is transitive,

(c) $f$ has sensitive dependence on initial conditions.
\end{definition}
It was shown by Banks et. al. \cite{Banks} that if $f$ satisfies
$(a)$ and $(b)$, then $f$ must have sensitive dependence on initial
conditions. Hence only the first two conditions of the definition
need to be verified.

Denote by $HC(H)$ and $DE(H)$ the set of all hypercyclic operators
and the set of all Devaney's chaotic operators on $H$ respectively.
Obviously, $DE(H)\subseteq HC(H)$.
\begin{proposition}\label{8}
$\overline{DE(H)}=\overline{HC(H)}\subseteq\overline{DC(H)^0}=\overline{LY(H)^0}$.
\end{proposition}
\begin{proof}
According to \cite{Her1} proposition 4 and \cite{Her2}, one can
obtain $\overline{DE(H)}=\overline{HC(H)}=\{T\in B(H)
;~\sigma_w(T)\cup\partial\mathbb{D}$ is connected,
$\sigma_0(T)=\emptyset$ and
ind$(\lambda-T)\geq0,~\lambda\in\rho_{s-F}(T)$\}. Following theorem
\ref{7}, we obtain the result.
\end{proof}

Next, we can see the set of all norm-unimodal operators is large in
the set of all distributionally chaotic operators.
\begin{theorem}\label{9}
$\overline{UN(H)}=\overline{DC(H)}=\overline{LY(H)}$,
$DC(H)^0=LY(H)^0\subseteq UN(H)$ and $\overline{UN(H)\backslash
DC(H)^0}=\overline{DC(H)\backslash
DC(H)^0}=\overline{LY(H)\backslash LY(H)^0}$.
\end{theorem}
\begin{proof}
First we prove $\overline{UN(H)}=\overline{DC(H)}=\overline{LY(H)}$.
Obviously,
$\overline{UN(H)}\subseteq\overline{DC(H)}=\overline{LY(H)}$, it
suffices to prove $E_1\cup E_2\subseteq\overline{UN(H)}$, where
$E_1,~E_2$ are denoted in theorem \ref{5}. We will show for any
$T\in E_1\cup E_2$ and $\epsilon>0$, there exists $C$ such that
$||C||<\epsilon$ and $T+C\in UN(H)$. But different to theorem
\ref{5}, we can not generally find a compact operator satisfying the
property.

If $T\in E_1$, then choose any $\lambda_0 \in
\partial\mathbb{D}\cap \sigma_{lre}(T)$. According to AFV theorem,
there exists a compact operator $K_1$ such that $||K_1||<\epsilon/2$
and
$$T+{K_1}=\begin{matrix}\begin{bmatrix}
\lambda_0I&*\\
&*\\
\end{bmatrix}&
\begin{matrix}
 H_0\\
  {H_0}^{\bot}\end{matrix}\end{matrix},$$
  where dim$H_0=\infty$. Let $$C_1=\begin{matrix}\begin{bmatrix}
0&\epsilon/2\\
&0&\epsilon/2\\
&&\ddots&\ddots
\end{bmatrix}&
\begin{matrix}
e_0\\
e_1\\
\vdots
\end{matrix}\end{matrix},$$ where $\{e_i\}_{i=0}^{\infty}$ is an ONB of $H_0$. Then $C_1$ is a Cowen-Douglas operator with
$||C_1||=\epsilon/2$. Let
$$\widetilde{C_1}=\begin{matrix}\begin{bmatrix}
C_1\\
&0\\
\end{bmatrix}&
\begin{matrix}
 H_0\\
  {H_0}^{\bot}\end{matrix}\end{matrix}.$$ Following theorem \ref{C-D-Operator},
we know
$$T+K_1+\widetilde{C_1}\in
UN(H),~ (||K_1+\widetilde{C_1}||<\epsilon).$$

Notice $K_1+\widetilde{C_1}$ is not compact. In fact, the operator
$T$ in $ E_1$ satisfying $\sigma(T)=\sigma_{lre}(T)$ and
$\sigma(T)\subseteq\mathbb{D}^{-}$ can not be perturbed into $UN(H)$
by compact operator, since $T+K$ is impossible to be norm-unimodal
(one can observe \cite{HBZ} for details).

If $T\in E_2$, then according to the first step in the proof of
theorem \ref{5}, we know $T$ is norm-unimodal. The first equation is
complete. The second inclusion is immediate obtained from the first
step in the proof of theorem \ref{6}.

Next we prove $\overline{UN(H)\backslash
DC(H)^0}=\overline{DC(H)\backslash
DC(H)^0}=\overline{LY(H)\backslash LY(H)^0}$. Clearly,
$\overline{UN(H)\backslash
DC(H)^0}\subseteq\overline{DC(H)\backslash
DC(H)^0}=\overline{LY(H)\backslash LY(H)^0}.$ We only need to show
$G_1\cup G_2\subseteq\overline{UN(H)\backslash DC(H)^0}$, where
$G_1,~G_2$ are denoted in theorem \ref{7}. Similar to the first step
of theorem \ref{5}, $G_1\subseteq UN(H)\backslash DC(H)^0$.

Let $T\in G_2$ and $\epsilon>0$. By AFV theorem there exists a
compact operator $K_1$ such that $||K_1||<\epsilon/3$ and
$$T+K_1=\begin{matrix}\begin{bmatrix}
\lambda_0I&*&*\\
&A&*\\
&&\lambda_0I
\end{bmatrix}&
\begin{matrix}
H_0\\
 {H_1}\\
  H_2\end{matrix}\end{matrix},$$ where $\lambda_0\in
 \partial\mathbb{D}\cap\sigma_{lre}(T),~\rho_{s-F}^{(+)}(A)\cap\partial\mathbb{D}=\emptyset$.
Let $N\in B(H_2)$ be an uniform infinite multiplicity normal
operator such that $\sigma(N)=\{z\in\mathbb{C};~|z|\leq\epsilon/3\}$
and $$B_{\epsilon}=\begin{matrix}\begin{bmatrix}
0&\epsilon/3\\
&0&\epsilon/3\\
&&\ddots&\ddots
\end{bmatrix}&
\begin{matrix}
e_0\\
 e_1\\
  \vdots\end{matrix}\end{matrix},
$$
where $\{e_i\}_{i=0}^{\infty}$ is an ONB of $H_0$.

Then $||N||=\epsilon/3$; $||B_{\epsilon}||=\epsilon/3, \
B_{\epsilon}\in
  B_1(\mathbb{D}_{\epsilon/3}) \  \text{and} \
\sigma(B_{\epsilon})={\mathbb{D}_{\epsilon/3}}^-.$

Hence
$$
T+K_1+\begin{matrix}\begin{bmatrix}
B_{\epsilon}\\
&0\\
&&0
\end{bmatrix}&
\begin{matrix}
\end{matrix}\end{matrix}+\begin{matrix}\begin{bmatrix}
0\\
&0\\
&&N\\
\end{bmatrix}&
\begin{matrix}
\end{matrix}\end{matrix}=\begin{matrix}\begin{bmatrix}
\lambda_0I+B_{\epsilon}&*&*\\
&A&*\\
&&\lambda_0I+N
\end{bmatrix}&
\begin{matrix}
\end{matrix}\end{matrix}.
$$

Notice
\begin{eqnarray*}
&&\rho_{s-F}^{(+)}(\begin{matrix}\begin{bmatrix}
\lambda_0I+B_{\epsilon}&*&*\\
&A&*\\
&&\lambda_0I+N
\end{bmatrix}&
\begin{matrix}
\end{matrix}\end{matrix})\cap\partial\mathbb{D}\\
&&=[\rho_{s-F}^{(+)}(A)\backslash\{z\in\mathbb{C};~|z-\lambda_0|\leq\epsilon/3\}]\cap
\partial\mathbb{D}\\
&&=\emptyset
\end{eqnarray*}
and $\lambda_0I+B_{\epsilon}$ is norm-unimodal, we obtain the
result.
\end{proof}

\begin{example}
Let $A\in B(H)$ satisfying
 \begin{align*}\left\{
\begin{array}{ll}\hspace{7mm}Ae_i=2e_{i-1},\ \hspace{25mm}  i\geq1,
\\[2mm]
\hspace{7mm}Ae_0=e_{-1},\\[2mm]
 \hspace{7mm}Ae_i=\frac{|i|}{|i|+1}e_{i-1},
\hspace{20mm} i\leq-1.
\end{array}\right.
\end{align*}
where $\{e_i\}_{i=-\infty}^{\infty}$ is ONB of $H$. Then $A$ is
norm-unimodal, but not in $DC(H)^0$.
\end{example}
\begin{proof}
For any $m\in\mathbb{N}$, $||A^i(e_m)||\geq2^i||e_m||,~1\leq i\leq
m$ and $\lim\limits_{n\rightarrow\infty}||A^n(e_m)||=0$. So $A$ is
norm-unimodal. Since $\sigma(A)=\{z\in\mathbb{C};~1\leq |z|\leq2\}$
implies $\partial\mathbb{D}\subseteq \sigma_{lre}(A)$, then $A$ is
not in $DC(H)^0$.
\end{proof}

\section{Connectedness}
The main purpose in this section is to discuss the connectedness for
the sets considered in section 2.
\begin{theorem}\label{10}
$DC(H)^0,~\overline{DC(H)^0}$, $\overline{DC(H)}$ and
$\overline{DC(H)\backslash DC(H)^0}$ ( i.e. $LY(H)^0$,
$\overline{LY(H)^0}$, $\overline{LY(H)}$ and
$\overline{LY(H)\backslash LY(H)^0}$ ) are all arcwise connected.
\end{theorem}
\begin{proof}
We show $DC(H)^0$ is arcwise connected, others are similar.

First step, for any $T\in DC(H)^0$, it can be connected to
$\widetilde{T}\in DC(H)^0$, where
$\sigma_{lre}(\widetilde{T})\cap\partial\mathbb{D}\neq\emptyset$. If
$\sigma_{lre}({T})\cap\partial\mathbb{D}\neq\emptyset$, then
obviously.

Let $\sigma_{lre}({T})\cap\partial\mathbb{D}=\emptyset$. Then
$\partial\mathbb{D}\subseteq\rho_{s-F}^{(+)}(T)$. Choose
$\lambda_0\in\sigma_{lre}(T)$, by AFV theorem there exists a compact
operator $K$ such that $||K||<\epsilon$ and
$$T+K=\begin{matrix}\begin{bmatrix}
\lambda_0I&*&*\\
&A&*\\
&&\lambda_0I\\
\end{bmatrix}&
\begin{matrix}
\end{matrix}\end{matrix}$$ where $\sigma(T)=\sigma(A), \sigma_{lre}(T)=\sigma_{lre}(A)
$ and $\text{ind}(\lambda-T)=\text{ind}(\lambda-A)$ for all
$\lambda\in\rho_{s-F}(A)$. Choose $\mu_0\in\partial\mathbb{D}$, let
$$\delta(t)=\begin{matrix}\begin{bmatrix}
\alpha(t)I&*&*\\
&A&*\\
&&\alpha(t)I\\
\end{bmatrix}&
\begin{matrix}
\end{matrix}\end{matrix},~1< t\leq2,$$ where $\alpha(t)=(t-1)(\mu_0-\lambda_0)+\lambda_0, 1< t\leq2$.
Define
\begin{align*}\beta(t)&=\left\{
\begin{array}{ll}\hspace{7mm}T+tK,\ \hspace{20mm}  0\leq t\leq1,
\\[2mm]
\hspace{7mm}\delta(t), \hspace{27mm}1< t\leq2.
\end{array}\right.
\end{align*}
Obviously,
$$\beta(0)=T,~\beta(2)=\widetilde{T}:=\begin{matrix}\begin{bmatrix}
\mu_0I&*&*\\
&A&*\\
&&\mu_0I\\
\end{bmatrix}&
\begin{matrix}
\end{matrix}\end{matrix},~\sigma_{lre}(\widetilde{T})\cap\partial\mathbb{D}=\{\mu_0\}\neq\emptyset$$
and  $\beta(t)$ is continuous on [0,2].

For any $1< t\leq2$,
$\rho^{(+)}_{s-F}(\delta(t))=\rho^{(+)}_{s-F}(A)\backslash
\{\alpha(t)\}$, then $
\rho^{(+)}_{s-F}(\delta(t))\cap\partial\mathbb{D}\neq\emptyset.$
Notice
$\rho^{(+)}_{s-F}(T+tK)=\rho^{(+)}_{s-F}(T)\supseteq\partial\mathbb{D}~\text{for
all}~0\leq t\leq1,$ we know $\{\beta(t);~0\leq t\leq2\}\subseteq
DC(H)^0$. The first step is complete.

Second step, we show for any $T,S\in
DC(H)^0,~\sigma_{lre}(T)\cap\partial\mathbb{D}\neq\emptyset,~\sigma_{lre}(S)\cap\partial\mathbb{D}\neq\emptyset$,
$T$ and $S$ can be connected.

Let
$\lambda_0\in\partial[\rho_{s-F}^{(+)}(T)\cap\partial\mathbb{D}]$ be
the point such that there exists $\theta_0>0$ s.t.
$\{\lambda_0e^{i\theta};~0<\theta<\theta_0\}\subseteq\rho_{s-F}^{(+)}(T)$.
Similarly, we can obtain
$\lambda_1\in\partial[\rho_{s-F}^{(+)}(S)\cap\partial\mathbb{D}]$
and $\theta_1>0$ such that
$\{\lambda_1e^{i\theta};~0<\theta<\theta_1\}\subseteq\rho_{s-F}^{(+)}(S)$.
Notice there exists $\theta^{'}$ such that
$\lambda_0=e^{i\theta^{'}}\lambda_1$, define $\widetilde{S}=
e^{i\theta'}S$. Then
$\lambda_0\in\partial[\rho_{s-F}^{(+)}(\widetilde{S})\cap\partial\mathbb{D}]$
and
$\{\lambda_0e^{i\theta};~0<\theta<\theta_1\}\subseteq\rho_{s-F}^{(+)}(\widetilde{S})$
(let $\Phi_0$ be the component of $\rho_{s-F}(\widetilde{S})$ which
contains $\{\lambda_0e^{i\theta};~0<\theta<\theta_1\}$).

By AFV theorem and theorem 3.48 \cite{HER}, for any $\epsilon>0$,
there exist compact operators $K_1,~K_2$ such that
$||K_1||<\epsilon,~||K_2||<\epsilon$ and
$$T+K_1=\begin{matrix}\begin{bmatrix}
\lambda_0I&C_1\\
&A_1\\
\end{bmatrix}&
\begin{matrix}
H_1\\
{H_1}^{\bot}\end{matrix}\end{matrix},~\widetilde{S}+K_2=\begin{matrix}\begin{bmatrix}
A_2&C_2\\
&\lambda_0I\\
\end{bmatrix}&
\begin{matrix}
{H_2}^{\bot}\\
H_2\end{matrix}\end{matrix},$$ where $
\sigma(A_1)=\sigma(T),~\sigma_{lre}(A_1)=\sigma_{lre}(T)$ and
ind$(\lambda-A_1)=\text{ind}(\lambda-T)$ for all
$\lambda\in\rho_{s-F}(T); $
$\sigma(A_2)=\sigma(\widetilde{S}),~\sigma_{lre}(A_2)=\sigma_{lre}(\widetilde{S})$,
ind$(\lambda-A_2)=\text{ind}(\lambda-\widetilde{S})$ for all
$\lambda\in\rho_{s-F}(\widetilde{S})$ and
$\text{minind}(\lambda-A_2)=0$ for all $\lambda\in\Phi_0$. Without
loss of generality, let $H_1={H_2}^{\bot}$. Define
\begin{align*}\gamma(t)&=\left\{
\begin{array}{ll}e^{it}S,\ \hspace{75mm}  0\leq t<\theta^{'}, \\[2mm]
\widetilde{S}+(t-\theta^{'})K_2, \hspace{57mm}  \theta^{'}\leq t\leq\theta^{'}+1,
\\[2mm]
(t-(\theta^{'}+1))(T+K_1)+(\theta^{'}+2-t)(\widetilde{S}+K_2), \hspace{6mm}  \theta^{'}+1< t<\theta^{'}+2,
\\[2mm]
T+[\theta^{'}+3-t]K_1, \hspace{48mm}  \theta^{'}+2\leq t\leq\theta^{'}+3.
\\
\end{array}\right.
\end{align*}
Obviously, $\gamma(0)=S,~\gamma(\theta^{'}+3)=T$ and $\gamma(t)$ is
continuous on [0, $\theta^{'}+3$]. We prove $\{\gamma(t);~0\leq
t\leq\theta^{'}+3\}\subseteq DC(H)^0$.

a) $\rho_{s-F}^{(+)}(e^{it}S)=e^{it}\rho_{s-F}^{(+)}(S),~0\leq
t<\theta^{'}$ and
$\rho_{s-F}^{(+)}(S)\cap\partial\mathbb{D}\neq\emptyset$ implies
$\{e^{it}S;~0\leq t<\theta^{'}\}\subseteq DC(H)^0$,

b)
$\rho_{s-F}^{(+)}(\widetilde{S}+(t-\theta^{'})K_2)=\rho_{s-F}^{(+)}(\widetilde{S})=e^{i\theta^{'}}\rho_{s-F}^{(+)}({S}),~\theta^{'}\leq
t\leq\theta^{'}+1$ and
$\rho_{s-F}^{(+)}(S)\cap\partial\mathbb{D}\neq\emptyset$ implies
$\{\widetilde{S}+(t-\theta^{'})K_2;~\theta^{'}\leq
t\leq\theta^{'}+1\}\subseteq DC(H)^0$,

c) For any given $\theta^{'}+1< t<\theta^{'}+2$, since
\begin{align*}
\rho_{s-F}^{(+)}([t-(\theta^{'}+1)]\lambda_0+[\theta^{'}+2-t]A_2)&=[t-(\theta^{'}+1)]\lambda_0+[\theta^{'}+2-t]\rho_{s-F}^{(+)}(A_2)
\\
&=\lambda_0+(\theta^{'}+2-t)(\rho_{s-F}^{(+)}(A_2)-\lambda_0),
\end{align*}
and
\begin{align*}
\rho_{s-F}^{(+)}([t-(\theta^{'}+1)]A_1+[\theta^{'}+2-t]\lambda_0)&=[t-(\theta^{'}+1)]\rho_{s-F}^{(+)}(A_1)+[\theta^{'}+2-t]\lambda_0
\\ &=\lambda_0+[t-(\theta^{'}+1)](\rho_{s-F}^{(+)}(A_1)-\lambda_0),
\end{align*}
we know there exists $\theta_t>0$ such that
\begin{align*}
&\{\lambda_0e^{i\theta};~0<\theta<\theta_t\} \\
\subseteq
&\rho_{s-F}^{(+)}([t-(\theta^{'}+1)]\lambda_0+[\theta^{'}+2-t]A_2)\cap\rho_{s-F}^{(+)}([t-(\theta^{'}+1)]A_1+[\theta^{'}+2-t]\lambda_0).
\end{align*}
For each $\lambda\in\{\lambda_0e^{i\theta};~0<\theta<\theta_t\}$,
notice $ \text{minind}(\mu-A_2)=0 \ \text{for} \ \mu\in\Phi_0, $ so
$\text{minind}([t-(\theta^{'}+1)]\lambda_0+[\theta^{'}+2-t]A_2-\lambda)=0$.
Moreover,
ind$([t-(\theta^{'}+1)]\lambda_0+[\theta^{'}+2-t]A_2-\lambda)>0$,
then we have
$[t-(\theta^{'}+1)]\lambda_0+[\theta^{'}+2-t]A_2-\lambda$ is epic.

Consequently
\begin{align*}
&\text{ind}([t-(\theta^{'}+1)](T+K_1)+[\theta^{'}+2-t](\widetilde{S}+K_2)-\lambda)
\\
=&\text{ind}([t-(\theta^{'}+1)]\lambda_0+[\theta^{'}+2-t]A_2-\lambda)+
\\ &\text{ind}([t-(\theta^{'}+1)]A_1+[\theta^{'}+2-t]\lambda_0-\lambda) \\ &>0.
\end{align*}
Hence
$\{(t-(\theta^{'}+1))(T+K_1)+(\theta^{'}+2-t)(\widetilde{S}+K_2);~\theta^{'}+1<
t<\theta^{'}+2\}\subseteq DC(H)^0$.

d)
$\rho_{s-F}^{(+)}(T+[\theta^{'}+3-t]K_1)=\rho_{s-F}^{(+)}(T),~\theta^{'}+2\leq
t\leq\theta^{'}+3$ and
$\rho_{s-F}^{(+)}(T)\cap\partial\mathbb{D}\neq\emptyset$ implies
$\{T+[\theta^{'}+3-t]K_1;~\theta^{'}+2\leq
t\leq\theta^{'}+3\}\subseteq DC(H)^0.$

Therefore, $\{\gamma(t);~0\leq t\leq\theta^{'}+3\}\subseteq
DC(H)^0$. The second step is complete.

Thus, $DC(H)^0$ is arcwise connected.
\end{proof}

\begin{example}
Let $B$ be backward unilateral shift. Then there exists an arc
$\alpha(t)$ in $DC(H)^0$ which connects $5B$ and $5B^2$.
\end{example}
\begin{proof}
Through easy compute, $\sigma(5B)=\sigma(5B^2)=5\mathbb{D}^-$ and
$\text{ind}(\lambda-5B)=1,~\text{ind}(\lambda-5B^2)=2$ for
$|\lambda|<5$. First there exist compact operators $K_1,~K_2$ such
that
$$
5B+K_1=\begin{matrix}\begin{bmatrix}
-5I&C_1\\
&A_1\\
\end{bmatrix}&
\begin{matrix}
H_1\\
 {H_1}^{\bot}\end{matrix}\end{matrix} \ \text{and} \
5B^2+K_2=\begin{matrix}\begin{bmatrix}
A_2&C_2\\
&5I\\
\end{bmatrix}&
\begin{matrix}
 {H_2}^{\bot}\\
H_2\end{matrix}\end{matrix},$$ where
$\sigma(A_1)=\sigma(5B),~\sigma_{lre}(A_1)=\sigma_{lre}(5B)$ and
$\text{ind}(\lambda-A_1)=\text{ind}(\lambda-5B)$ for all
$\lambda\in\rho_{s-F}(5B)$;
$\sigma(A_2)=\sigma(5B^2),~\sigma_{lre}(A_2)=\sigma_{lre}(5B^2)$ and
$\text{ind}(\lambda-A_2)=\text{ind}(\lambda-5B^2)$ for all
$\lambda\in\rho_{s-F}(5B^2)$. Without loss of generality, let
$H_1={H_2}^{\bot}$. Define
\begin{align*}
\alpha(t)&=\left\{
\begin{array}{ll}\hspace{7mm}5B+(1+t)K_1,\hspace{35mm}  -1\leq t\leq0,
\\[2mm]
\hspace{7mm}\delta(t), \hspace{55mm}0< t<1,\\[2mm]
\hspace{7mm} 5B^2+(2-t)K_2,\hspace{35mm}  1\leq t\leq2,
\end{array}\right.
\end{align*}
where $\delta(t)=\begin{matrix}\begin{bmatrix}
-5(1-t)+tA_2&(1-t)C_1+tC_2\\
&(1-t)A_1+5t\\
\end{bmatrix}&
\begin{matrix}
\end{matrix}\end{matrix},~0< t<1$.

Obviously, $\alpha(-1)=5B,~\alpha(2)=5B^2$ and $\alpha(t)$ is
continuous on [-1,2]. It suffices to show $\alpha(t)\in DC(H)^0$ for
any $-1\leq t\leq2$.

(1) $\rho_{s-F}^{(+)}(5B+(1+t)K_1)=\rho_{s-F}^{(+)}(5B),~-1\leq
t\leq0$ and $\partial\mathbb{D}\subseteq\rho_{s-F}^{(+)}(5B)$
implies $\{5B+(1+t)K_1;~-1\leq t\leq0\}\subseteq DC(H)^0$.

(2) $\rho_{s-F}^{(+)}(\delta(t))
=[\rho_{s-F}^{(+)}(-5(1-t)+tA_2)]\cup
[\rho_{s-F}^{(+)}((1-t)A_1+5t)]
=\{-5(1-t)+5t\mathbb{D}\}\cup\{5(1-t)\mathbb{D}+5t\},~0< t<1$
implies $\{\delta(t);~0<t<1\}\subseteq DC(H)^0$.

One can read spectrum properties from the picture as follows. The
number 1 or 2 in the picture means the index of the open disk which
it lies respectively. We choose five moments.

\scalebox{0.38}[0.38]{\includegraphics[165,870][0,0]{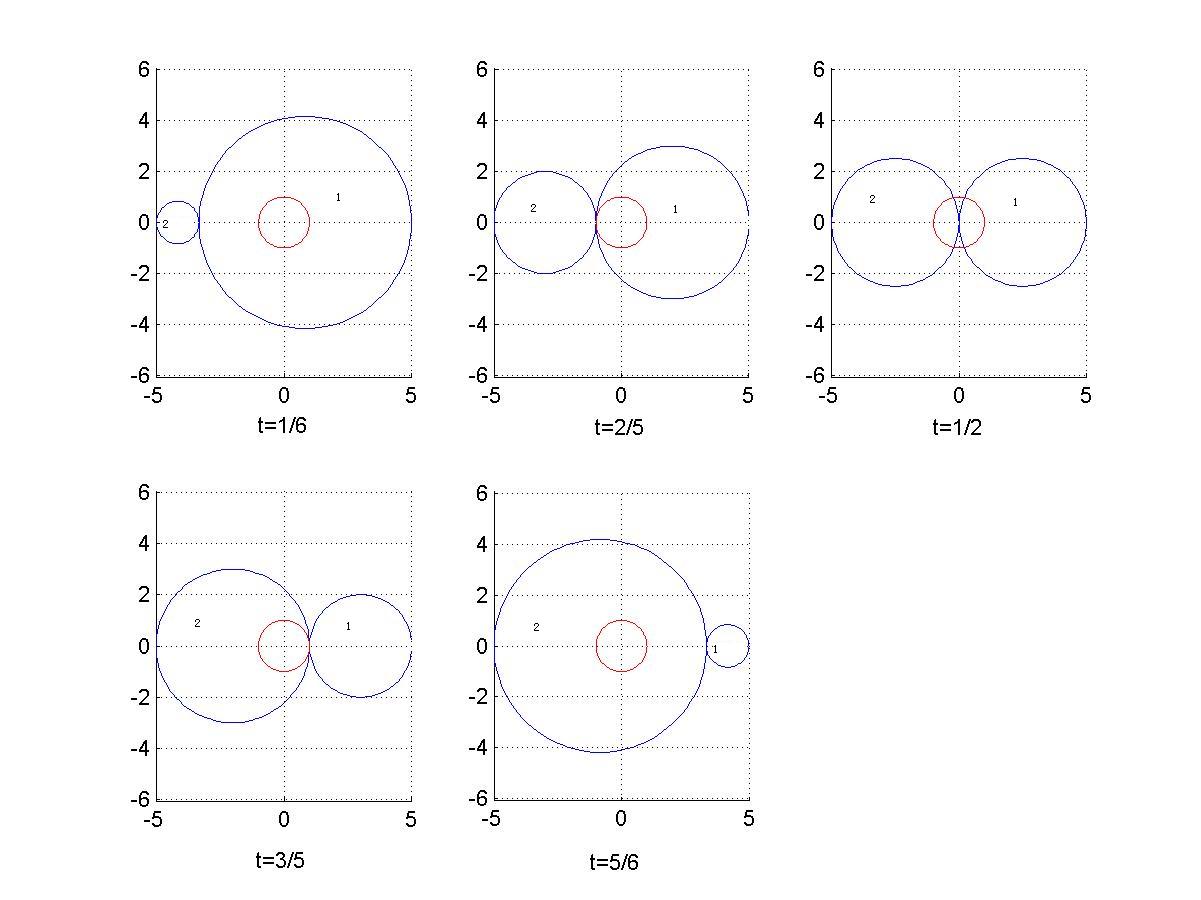}}
\vspace*{120mm}

(3) $\rho_{s-F}^{(+)}(5B^2+(2-t)K_2)=\rho_{s-F}^{(+)}(5B^2),~1\leq
t\leq2$ and $\partial\mathbb{D}\subseteq\rho_{s-F}^{(+)}(5B^2)$
implies $\{5B^2+(2-t)K_2;~1\leq t\leq2\}\subseteq DC(H)^0$.

Hence $\{\alpha(t),~-1\leq t\leq2\}\subseteq DC(H)^0$.
\end{proof}

{\bf Acknowledgement.} A large part of this article was developed
during the seminar on operator theory and dynamical system held at
the University of Jilin in China. The authors are deeply indebted to
Cao Yang, Cui Puyu, Hu Xi etc..

\end{document}